\documentclass[reqno]{amsart}

\usepackage{amsmath}
\usepackage{amssymb}
\usepackage{amsfonts}
\usepackage{graphicx}
\usepackage{amsthm}
\usepackage{enumerate}
\usepackage[mathscr]{eucal}
\usepackage{lscape}
\usepackage{dsfont}
\usepackage{color}
\usepackage{mathtools}

\usepackage{setspace}
\newtheorem{theor}{Theorem}[section]

\newcommand{\Hm}{\mathcal{H}_m}

\newcommand{\C}{\mathbb{C}}

\newcommand{\K}{K\"{a}hler\ }

\newcommand{\Ric}{\textrm{Ric}}

\begin{document}
\title[On the balanced condition for the Eguchi-Hanson metric]{On the balanced condition for the Eguchi-Hanson metric}

\author[Francesco Cannas Aghedu]{{\bfseries Francesco Cannas Aghedu}} 

\address{Dipartimento di Matematica e Informatica, Universit\`a di Cagliari\\Via Ospedale 72, 09124 Cagliari (Italy) }
\email{fcannasaghedu@unica.it}

\subjclass[2010]{53C55; 58C25;  58F06} 
\keywords{K\"ahler manifolds; projectively induced metrics; balanced metric; Ricci-flat metrics;  Eguchi-Hanson metric.}

\begin{abstract}
Let $g_{EH}$ be the Eguchi-Hanson metric on the blow-up of $\C^2$ at the origin. In this paper we show that $mg_{EH}$ is not balanced for any positive integer $m$.
\end{abstract}
 
\maketitle
\tableofcontents
\newpage
\section{Introduction}
The characterization of K\"ahler-Einstein projectively induced metrics is an important open problem in the field of \K geometry.

Here a \K metric $g$ on a complex manifold $M$ is said to be \textit{projectively induced} if there exists a holomorphic and isometric (i.e. K\"ahler) immersion of $(M,g)$ into the complex projective space $(\C P^N,g_{FS}), N \leq +\infty$, endowed with the Fubini-Study metric $g_{FS}$, the metric whose associated \K form is given in homogeneous coordinates by $\omega_{FS}= \frac{i}{2\pi}\partial\bar{\partial}\log(|Z_0|^2+ \cdots + |Z_N|^2)$. The reader is referred to the recent book \cite{libdiast} for an updated report  on \K immersions into complex projective spaces and, more generally,  complex space forms.

In particular, in a recent paper A. Loi, F. Salis and F. Zuddas  \cite{LSZ} study  projectively induced Ricci-flat metrics and  they prove that  the Eguchi-Hanson metric $g_{EH}$ on the blow up of $\C^2$ at the origin (see Section \ref{section2} below)  is not projectively induced. Moreover, in the same paper they conjecture that $mg_{EH}$ is not projectively induced for any positive integer $m$, and they give evidence of this fact for small values of  the integer $m$.

The aim of this paper is to provide the validity of the conjecture by restricting it to an interesting class of projectively induced metrics, namely  the balanced metrics in the sense of  Donaldson (see Section \ref{section3}  below for details).\\

Our main result  is then  the following:
\begin{theor}\label{Th1}
The metric $m g_{EH}$ is not balanced for any positive integer $m$. 
\end{theor}
The paper consists of two other sections. Section \ref{section2} contains some basic  facts on the  Eguchi-Hanson metric and  in Section \ref{section3}, after a brief introduction to balanced  metrics, we prove  Theorem \ref{Th1}.

\vskip 0.3cm

\noindent 
The author would like to thank Andrea Loi and Michela Zedda for their really useful considerations, comments and remarks concerning various aspects of this work.

\section{The Eguchi-Hanson metric}\label{section2}
Let us first consider the blow-up of $\C^2$ at the origin. Define a subset $\tilde \C^2$ of $\C^2 \times \mathbb{C}P^{1}$ as the space of pairs $(z,\ell)$, where $\ell \in \mathbb{C}P^{1}$ and $z$ is a point on the line corresponding to  $\ell$ in $\C^2$, i.e:
\[
\tilde \C^2 = \{ (z_1, z_2, [t_1, t_2]) \in \C^2 \times \mathbb{C}P^{1} \ : \ t_1z_{2}-t_2z_{1} = 0\}.
\]

The blow-up $\tilde \C^2$ is a closed submanifold of $\C^2 \times \mathbb{C}P^{1}$ of complex dimension $2$. 	
A system of charts for $\tilde{\C}^2$ is given as follows: for  $j =1,2$ we take 
\[
\tilde U_j = (\C^2 \times U_j) \cap \tilde{\C}^2,
\]
where $U_j=\{t_j \neq 0\}$, for $j =1, 2$,  are open subsets of $\mathbb{C}P^{1}$. Then we have two coordinate maps
\begin{equation*}
\begin{split}
\varphi_1 & : \tilde{U}_1 \rightarrow \mathbb{C}^2,\, \left(z_1,z_2,[t_1,t_2]\right) \mapsto \left(z_1,\frac{t_2}{t_1}\right) \\
\varphi_2 & : \tilde{U}_2 \rightarrow \mathbb{C}^2,\, \left(z_1,z_2,[t_1,t_2]\right) \mapsto \left(\frac{t_1}{t_2},z_2\right)
\end{split}
\end{equation*}
having as inverses the parametrization maps defined, respectively, by
\begin{equation}\label{chartsBUmap2}
\begin{split}
\varphi_1^{-1} & : \mathbb{C}^2 \rightarrow \tilde{U}_1,\, (w_1,w_2) \mapsto (w_1,w_1w_2,[1,w_2]) \\
\varphi_2^{-1} & : \mathbb{C}^2 \rightarrow \tilde{U}_2,\, (w_1,w_2) \mapsto (w_1w_2,w_2,[w_1,1]).
\end{split}
\end{equation}

There are two projection maps
\begin{align*}
p_1 & : \tilde \C^2 \rightarrow \C^2 \\
p_2 & : \tilde \C^2 \rightarrow \mathbb{C}P^{1} 
\end{align*}
given by the restriction to $\tilde \C^2$ of the canonical projections of $\C^2 \times \mathbb{C}P^{1}$. One can prove (see \cite{McDuff}) that $p_2$ induces on  $\tilde{\C}^2$ the structure of complex line bundle, whose fibre over $[t_1,t_2] \in \mathbb{C}P^{1}$ is the corresponding line $\{(\lambda t_1, \lambda t_2)\, |\, \lambda \in \C\}$ in $\C^2$. In other words, this is the universal line bundle over $\mathbb{C}P^{1}$. Observe that $p_1$ is bijective when restricted to $p_1^{-1}(\C^2 \setminus \{0\})$, while
\[
p_1^{-1}(0) = \{(z, [t]) \in \tilde \C^2\ | \ z =  0 \} \simeq \mathbb{C}P^{1}.
\]
Thus we may think of $\tilde{\C}^2$ as obtained from $\C^2$ by replacing the origin $0$ by the space of all lines in $\C^2$ through $0$. The manifold $p_1^{-1}(0)$ is called the \textit{exceptional divisor}, and we will denote it by $H$. So, the restriction 
\[
p_r := {p_1}_{|\tilde{\C}^2 \setminus H}: \tilde{\C}^2 \setminus H \rightarrow \C^2 \setminus \{ 0 \}, \ \ (z, [t]) \mapsto z
\]
is a biholomorphism, having as inverse
\[
\C^2 \setminus \{ 0 \} \rightarrow \tilde{\C}^2 \setminus H, \ \ (z_1,z_2) \mapsto \left( z_1,z_2, \left[z_1,z_2\right] \right).
\]

Take now on $\C^2 \setminus \{0\}$ the $(1,1)$-form given by 
\begin{equation}
\omega = \frac{i}{2\pi} \partial \bar \partial(\sqrt{|z|^4+1}+\log|z|^2-\log(1+\sqrt{|z|^4+1})).
\label{WS}
\end{equation}
 We claim that the pull-back $p_r^*(\omega)$ of $\omega$, a priori defined only on $\tilde{\C}^2 \setminus H$, extends in fact to all $\tilde \C^2$.

The pull-back $p_r^*(\omega)$ is given in the coordinates (\ref{chartsBUmap2}) by
\[
p_r^*\omega  = \frac{i}{2\pi}\partial\bar{\partial}\left(\sqrt{|w_1|^4(1+|w_2|^2)^2+1} + \log\left( \frac{1+|\omega_2|^2}{1+\sqrt{|w_1|^4(1+|w_2|^2)^2+1}}\right)\right),
\]
on $\tilde{U}_1 \setminus H$, and
\[
p_r^*\omega  = \frac{i}{2\pi}\partial\bar{\partial}\left(\sqrt{|w_2|^4(1+|w_1|^2)^2+1} + \log\left( \frac{1+|\omega_1|^2}{1+\sqrt{|w_2|^4(1+|w_1|^2)^2+1}}\right)\right),
\]
on  $\tilde{U}_2 \setminus H$. This shows that $p_r^*(\omega)$ extends to the whole $\tilde \C^2$, as claimed.  On $\tilde{\C}^2 \setminus H$, clearly, this form is given in local coordinates by \eqref{WS}. The metric associated to  \eqref{WS} is known in literature as the Eguchi-Hanson metric and denoted here by $g_{EH}$. The form \eqref{WS} is denoted here by $\omega_{EH}$. It is not hard to see  that $g_{EH}$ is a complete Ricci-flat \K metric (cf. \cite{EH} and \cite{LSZ}).
\section{Balanced metrics and the  proof of Theorem \ref{Th1}}\label{section3}
Let $g$ be a \K metric on an $n$-dimensional complex manifold $M$. Assume that
there exists a holomorphic line bundle $L$ over $M$ such that $c_1(L) = [\omega]_{dR}$, where $\omega$ is the K\"ahler form associated to $g$, $c_1(L)$ denotes the first Chern class of $L$ and $[\omega]_{dR}$ is the second De-Rham cohomology class of $\omega$.  A necessary and sufficient condition for the existence of such an  $L$ is that $\omega$ is an integral \K form.

Let $m\geq 1$ be a positive integer and let $h_m$ be an Hermitian metric on $L^m$, where $L^m$ is the $m$-tensor power of $L$, such that its Ricci curvature $\Ric(h_m) = m\omega$\footnote{$\Ric(h_m)$ is the two-form on $M$ whose local expression is given by
	\[
	\Ric(h_m)= -\frac{i}{2\pi}\partial\bar{\partial}\log h_m(\sigma(x),\sigma(x))
	\]
	for a trivializing holomorphic section $\sigma : U \rightarrow L^m\setminus \{0\}$.}.  In the quantum mechanics terminology $L^m$ is called the \textit{prequantum line bundle}, the pair $(L^m, h_m)$ is called a \textit{geometric quantization} of the K\"ahler manifold $(M,m\omega)$ and $\hbar= m^{-1}$ plays the role of Planck's constant
(see e.g. \cite{AreLoi}). 

Consider the separable complex Hilbert space $\Hm$ consisting of global
holomorphic sections $s$ of $L^m$ which are bounded with respect to
\[
\left\langle s,s \right\rangle_{h_m}=||s||^2_{h_m}= \int_{M} h_m(s(x),s(x))\frac{\omega^n}{n!}.
\]
Let $s_j,\, j=0,\ldots, d_m (\dim \Hm = d_m+1\leq \infty)$ be an orthonormal basis of $(\Hm,\left\langle \cdot, \cdot \right\rangle_{h_m})$ and consider the function $\epsilon_{mg}$ on $M$ given by:
\[
\epsilon_{mg}(x)= \sum_{j=0}^{d_m}h_m(s_j(x),s_j(x)).
\]
The notation emphasizes that the function $\epsilon_{mg}$ depends only on the \K metric $mg$ and not on the orthonormal basis chosen and the Hermitian metric $h_m$. Clearly if $M$ is compact $\Hm=H^0(L^m)$, where $H^0(L^m)$ is the (finite dimensional) space of global holomorphic sections of $L^m$. 

\noindent The function $\epsilon_{mg}$ has appeared in literature under different names. The first one was  $\eta$\textit{-function} of Rawnsley in \cite{Rawnsley} later renamed as $\theta$\textit{-function} in \cite{CahGut}. 

It is well known (see \cite{CahGut} and \cite{Rawnsley}) that if the function $\epsilon_{mg}$ is a constant different from zero, for a suitable $m$, then $m g$ is projectively induced via the coherent states map 
\[
\varphi_m : M \rightarrow \C P^{d_m},\, x \mapsto [s_0(x),\ldots, s_j(x),\ldots].
\]
In fact the relation between this map and the function $\epsilon_{mg}$ can be read in the following formula due to Rawnsley (see \cite{Rawnsley}):
\begin{equation}
\varphi_m^*(\omega_{FS}) = \frac{i}{2\pi} \partial \bar{\partial} \log \sum _{j=0}^{d_m} |s_j(x)|^2 = m \omega+ \frac{i}{2\pi} \partial \bar{\partial} \log \epsilon_{mg}.  
\label{E10}
\end{equation}
Therefore, $\epsilon_{mg}$ measures the obstruction for the \K form $m \omega$ to be projectively induced via the coherent states map $\varphi_m$. 

A metric $g$ on $M$ is called {\em balanced}  if  $\epsilon_g$ is a positive constant.
The definition of balanced metrics was originally given by Donaldson \cite{Donaldson} in the case of compact polarized \K manifolds $(M,g)$ and generalized in \cite{AreLoi3} to the non compact case (see also \cite{CucLoi}, \cite{Englis2}, \cite{GreLoi}, \cite{LM12}, \cite{LZ12}). Notice $g$ is balanced does not imply $mg$ is balanced. \\

In order to prove Theorem \ref{Th1} consider the holomorphic line bundle $L \rightarrow (\tilde{\C}^2,\omega_{EH})$ such that $c_1(L)=[\omega_{EH}]_{dR}$. Such a  line bundle exists since $\omega_{EH}$ is integral. Moreover, $L$ is unique, up to isomorphisms of line bundle, since $\tilde{\C}^2$ is simply-connected.  It is straightforward to verify that the holomorphic line bundle $L^m  \rightarrow \tilde{\C}^2$, equipped with the hermitian structure 
\[
h_m(\sigma(x),\sigma(x))= e^{-m\sqrt{|z|^4+1}}\left(\frac{1+\sqrt{|z|^4+1}}{|z|^2} \right)^m |q|^2,
\]
defines a geometric quantization of $(\tilde{\C}^2,m\omega_{EH})$, where $m$ is a positive natural number and $\sigma: U \subset \tilde \C^2\setminus H \rightarrow L^m \setminus \{0\},\, x \mapsto (z,q) \in \C^2\setminus \{0\} \times \C$ is a trivialising holomorphic section.
Moreover, since $L^m_{|\C^2\setminus \{0\}}$ is equivalent to the trivial bundle $\C^2\setminus \{0\} \times \C$, one can find  a natural bijection between the space $H^0(L^m)$ of global holomorphic sections of $L^m$  and the space of  holomorphic functions on $\C^2$ vanishing at the origin with order greater or equal than $m$ 
(see, e.g.  \cite{GH}, Chapter 1). This bijection takes $s\in H^0(L^m)$ to the holomorphic function $f_s$ on $\C^2$ obtained by restricting $s$ to $\tilde{\C}^2 \setminus H \simeq \C^2 \setminus \{0\}$. Moreover, since $H$ has zero measure in $\tilde{\C}^2$, one gets 
\begin{equation}\label{intnorm}
\begin{split}
\left\langle s,s \right\rangle_{h_m} & = \int_{\tilde{\C}^2} h_m(s(x),s(x))\frac{\omega_{EH}^2}{2!}= \\
 & = \int_{\C^2 \setminus \{0\}} e^{-m\sqrt{|z|^4+1}}\left(\frac{1+\sqrt{|z|^4+1}}{|z|^2} \right)^m  |f_s(z)|^2d\mu (z)<\infty,
\end{split}
\end{equation}
where $d\mu (z)= \left(\frac{i}{2\pi} \right)^2 dz_1 \wedge d\bar{z}_1\wedge dz_2 \wedge d\bar{z}_2$, I would recall that the metric $\omega_{EH}$ is Ricci-flat. 
Therefore, in this case, the inclusion $\mathcal{H}_m\subseteq H^0(L^m)$ is indeed an equality, namely $\mathcal{H}_m=H^0(L^m)$. 
\newline

We are now ready to prove Theorem \ref{Th1}.
\begin{proof}[Proof of Theorem\ref{Th1}]
By passing to polar coordinates  $z_1=\rho_1 e^{i\vartheta_1}, z_2=\rho_2 e^{i\vartheta_2}$ with $\rho_1,\rho_2\in (0,+\infty), \vartheta_1,\vartheta_2 \in (0,2\pi)$  one  easily sees that the monomials $\{z_1^{j}z_2^{k}\}_{j+k \geq m}$ is a complete  orthogonal system for the Hilbert space $(\Hm,\left\langle \cdot, \cdot \right\rangle_{h_m})$.

Moreover, by (\ref{intnorm}),  
	\[
	||z_1^jz_2^k||^2_{h_m}   =4\int_0^{+\infty}\int_0^{+\infty} e^{-m\sqrt{(\rho_1^2+\rho_2^2)^2+1}}\left(\frac{1+\sqrt{(\rho_1^2+\rho_2^2)^2+1}}{\rho_1^2+\rho_2^2} \right)^m \rho_1^{2j+1}\rho_2^{2k+1}d\rho_1d\rho_2. 	
	\]
	With the substitution $\rho_1=r\cos\theta$, $\rho_2=r\sin\theta$, $0<r<+\infty$, $0<\theta <\frac{\pi}{2}$ one finds a product of one variable integrals
	\[
	||z_1^jz_2^k||^2_{h_m} = 4 \int_0^{\frac{\pi}{2}}(\cos\theta )^{2j+1}(\sin\theta)^{2k+1}d\theta\cdot
	\int_0^{+\infty} e^{-m\sqrt{r^4+1}}(1+\sqrt{r^4+1})^mr^{2(j+k-m+1)+1}\,dr.
	\]
	For the first integral we have (see \cite{AbrSteg} 6.1.1, page 255)
	\[
	\int_0^{\frac{\pi}{2}}(\cos\theta )^{2j+1}(\sin\theta)^{2k+1}d\theta=\frac{\Gamma (j+1)\Gamma (k+1)}{2\Gamma (j+k+2)}=
	\frac{j!k!}{2(j+k+1)!}
	\]
	Hence one gets  
	\begin{equation}
	||z_1^jz_2^k||^2_{h_m} = \frac{2j!k!}{(j+k+1)!}
	\int_0^{+\infty} e^{-m\sqrt{r^4+1}}(1+\sqrt{r^4+1})^mr^{2(j+k-m+1)+1}\,dr.
    \label{E0}
	\end{equation}

	Suppose that there exists a positive integer $m_0$ such that $m_0 \omega_{EH}$ is balanced. Therefore by \eqref{E10}, we have 
	\[
	\frac{i}{2\pi} \partial \bar{\partial} \log \sum_{j+k \geq m_0} \left| \frac{z_1^jz_2^k}{||z_1^jz_2^k||_{h_{m_0}}}\right|^2 =m_0 \omega_{EH} =  m_0\frac{i}{2\pi} \partial \bar{\partial} \Phi,
	\]
	where $\Phi$ is the \K potential associated to $\omega_{EH}$. Then there exists a holomorphic function $f$ on $\C^2$ such that 
	\[
	\log\left(\sum_{j+k \geq m_0} \left| \frac{z_1^jz_2^k}{||z_1^jz_2^k||_{h_{m_0}}}\right|^2 e^{-m_0 \Phi}\right) = \mathfrak{R}(f),
	\]
	where $\mathfrak{R}(f)$ denotes the real part of $f$. By radiality, $f$ is forced to be a constant, and so
	\begin{equation}
	\sum_{j+k \geq m_0} \left| \frac{z_1^jz_2^k}{||z_1^jz_2^k||_{h_{m_0}}}\right|^2 = C e^{m_0 \Phi}, 
	\label{E1}
	\end{equation}
	where $C$ is a real positive constant. A straightforward calculation shows that the series expansion of $e^{m_0 \Phi}$ at $(z_1,z_2)=(0,0)$ is given by
	\begin{equation}
	\left(\frac{e}{2}\right)^{m_0}\sum_{s=0}^{m_0}\binom{m_0}{s}|z_1|^{2(m_0-s)}|z_2|^{2s} + \frac{m_0}{4}\left(-\frac{e}{2}\right)^{m_0} \sum_{s=0}^{m_0+2}\binom{m_0+2}{s}|z_1|^{2(m_0+2-s)}|z_2|^{2s}+ o(|z|^6).
	\label{E2}
	\end{equation}
	From \eqref{E1} and \eqref{E2}, we find
	\begin{equation}
	\frac{|z_1|^{2m_0}}{||z_1^{m_0}||^2_{h_{m_0}}} = C  \left(\frac{e}{2}\right)^{m_0} m_0|z_1|^{2m_0}
	\label{E3}
	\end{equation}
	for $(j,k)=(m_0,0)$, and 
	\begin{equation}
	\frac{|z_1|^{2(m_0+2)}}{||z_1^{m_0+2}||^2_{h_{m_0}}} = C \frac{m_0}{4}\left(-\frac{e}{2}\right)^{m_0}(m_0+2)|z_1|^{2(m_0+2)}
	\label{E4}
   \end{equation}
	for $(j,k)=(m_0+2,0)$. So if $m_0$ is odd, \eqref{E4} yields a contradiction. When $m_0$  is even, by comparing  \eqref{E3}-\eqref{E4}, we must have 
	\begin{equation}
	C =  \left(\frac{2}{e}\right)^{m_0}\frac{1}{m_0||z_1^{m_0}||^2_{h_{m_0}}} = \left(\frac{2}{e}\right)^{m_0}\frac{4}{m_0(m_0+2)||z_1^{m_0+2}||^2_{h_{m_0}}}.
	\label{E7}
	\end{equation}
	From \eqref{E0}, by integrating, we find
	\begin{equation}
	||z_1^{m_0}||^2_{h_{m_0}}= \frac{1}{m_0^2(m_0+1)}\left(\frac{e}{m_0}\right)^{m_0}\left(\Gamma(m_0+2,2m_0)-m_0\Gamma(m_0+1,2m_0)\right),  
	\label{E5}
	\end{equation}
	and
	\begin{equation}
	\begin{split}
    ||z_1^{m_0+2}||^2_{h_{m_0}}= \frac{1}{m_0^4(m_0+3)}\left(\frac{e}{m_0}\right)^{m_0} \left(\right. \Gamma(m_0+4,2m_0) & -3m_0\Gamma(m_0+3,2m_0) \\ 
    & + 2m_0^2 \Gamma(m_0+2,2m_0)\left.\right) ,  
	\end{split}
	\label{E6}
	\end{equation}
	where $\Gamma (a,b) = \int_{b}^{\infty} t^{a-1}e^{-t}dt$ is the incomplete Gamma function. By substituting \eqref{E5} in \eqref{E7}, we find
	\begin{equation}
	C= \left( \frac{2 m_0}{e^2}\right)^{m_0} \frac{m_0(m_0+1)}{\Gamma(m_0+2,2m_0)-m_0\Gamma(m_0+1,2m_0)}
	\label{E8}
	\end{equation}
	and by substituting \eqref{E6} in \eqref{E7}, one gets
	\begin{equation}
	C=\left( \frac{2 m_0}{e^2}\right)^{m_0}\frac{4m_0^3(m_0+3)}{(m_0+2)(\Gamma(m_0+4,2m_0) -3m_0\Gamma(m_0+3,2m_0) + 2m_0^2 \Gamma(m_0+2,2m_0))}.
	\label{E9}
	\end{equation}
	Consider now the real function
	\begin{equation}
	\begin{split}
	f(x)= &  \frac{x(x+1)}{\Gamma(x+2,2x)-x\Gamma(x+1,2x)} + \\
	        & - \frac{4x^3(x+3)}{(x+2)(\Gamma(x+4,2x) -3x\Gamma(x+3,2x) + 2x^2 \Gamma(x+2,2x))}
	\end{split}
	\label{E11}
	\end{equation}
	for $x \in [0,+\infty)$. The graph of the  function $y=f(x)$ is given in  Figure \ref{fig:plot} below (the fact that $\lim_{x \to \infty} f(x) = 0$ follows by the asymptotic series representation for the incomplete Gamma function given in \cite{Amore}). So, the value of  the costant $C$ in \eqref{E8} is equal to the value in \eqref{E9} if and only if $m_0=0$  in contrast with the positivity of $m_0$, yielding the desired contradiction. The proof of the theorem is complete.
	\begin{figure}[h]
		\centering
		\includegraphics[width=0.7\linewidth]{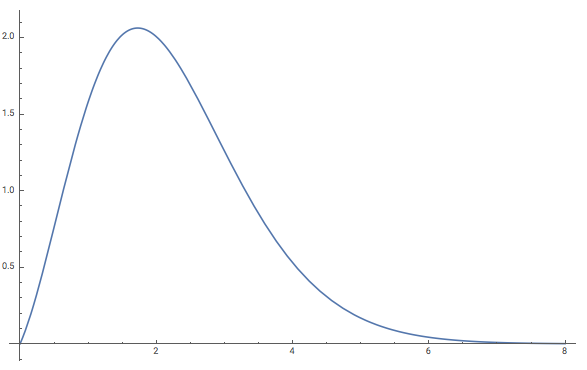}
		\caption{$y=f(x)$}
		\label{fig:plot}
	\end{figure}
	\end{proof}

\end{document}